\documentclass{article}
\usepackage{graphicx}

\usepackage[fleqn]{amsmath}

\usepackage{amsthm}
\usepackage{amssymb}
\usepackage{amsbsy}
\usepackage{amsfonts}
\usepackage{mathrsfs}
\usepackage{amsmath} 
\usepackage[all]{xy}
\usepackage{amstext}
\usepackage{amscd}
\usepackage[dvips]{epsfig}
\usepackage{psfrag}
\usepackage{enumerate}
\usepackage{flafter}
\allowdisplaybreaks

\usepackage{tikz}

\theoremstyle{plain}
\newtheorem{thm}{Theorem}[section]
\newtheorem{prop}[thm]{Proposition}
\newtheorem{cor}[thm]{Corollary}
\newtheorem{lem}[thm]{Lemma}
\theoremstyle{definition}
\newtheorem{exa}[thm]{Example}

\newtheorem{defn}[thm]{Definition}

\def\Ker{\mathop{\mathrm{Ker}}\nolimits}

\newcommand{\lra}{\longrightarrow}
\newcommand{\ra}{\rightarrow}

\newcommand{\Z}{{\Bbb Z}}

\newcommand{\pc}[2]{\mbox{$\begin{array}{c}
\includegraphics[scale=#2]{#1.eps}
\end{array}$}}
\begin{document}
\large
\begin{center}{\bf\Large Milnor invariants via unipotent Magnus embeddings}
\end{center}
\vskip 1.5pc

\begin{center}{\Large Hisatoshi Kodani\footnote{
E-mail address: {\tt kodani@mpim-bonn.mpg.de}} and Takefumi Nosaka\footnote{
E-mail address: {\tt nosaka@math.titech.ac.jp}
}}\end{center}
\vskip 1pc
\begin{abstract}\baselineskip=12pt \noindent
We reconfigure the Milnor invariant of links in terms of central group extensions and unipotent Magnus embeddings.
We also develop a diagrammatic computation of the invariant
and compute the first non-vanishing invariants of Milnor link and of several other links.
Moreover, we refine the original Milnor invariants of higher degree. 
\end{abstract}

\begin{center}
\end{center}

\large
\baselineskip=16pt
\section{Introduction.}In landmark papers in the 1950's \cite{Mil1,Mil2}, Milnor introduced higher order linking numbers.
The Milnor invariants have become increasingly well understood from a topological perspective,
e.g., in terms of higher Massey products, link concordance, nilpotent DGA, and finite type invariants (see \cite{Coc,Hil,HM,IO,Stein} and references therein).
However, there are only few methods for computing the invariants that are applicable to arbitrary links with many crossings.
In fact, the original definition strongly depends on group presentation of the individual longitudes of each link component, and
it seems difficult to compute the invariants from
the nilpotent noncommutativity appearing in these papers. 
In \S \ref{Srev}, we give a review of Milnor invariants and point out five difficulties in the previous computations.

This paper develops a diagrammatic computation of Milnor invariants; see Theorem \ref{thm1}.
The point here is to reformulate the invariant in terms of central group extensions, 
and to 
use the unipotent Magnus expansion \cite{GG}, which deals with nilpotent groups as a matrix group over a commutative ring $\Omega_m$ (see \S \ref{Semb} for the definition and properties).
Moreover, this paper is inspired by the quandle cocycle invariant in \cite{CEGS} (see also Appendix \ref{App2}), which sometimes provides a computation without longitudes.
To summarize, our computations do not require an explicit presentation for the longitudes, and therefore are compatible with computer programs (see \S\S \ref{Sthm}--\ref{Sexa} for details).
In fact, Section \ref{Sexa} gives some examples of such computations.

As a corollary, 
we refine the higher Milnor $\bar{\mu}$-invariants; see \S \ref{SHigher}.
The original definition has indeterminacy 
modulo certain ideals ``$\Delta(I)$"$\subset \Z$;
the rational parts always vanish, and it happens that all the higher $\bar{\mu}$-invariants are zero.
However, we alternatively introduce weaker ideals of the ring $\Omega_m$, and modify the higher invariants modulo these ideals (Proposition \ref{oo}), and show a universality result (Theorem \ref{thm5}).
We see (\S \ref{Sb31}) that the refined invariants are rationally non-trivial and
can detect some links.

In conclusion, the (higher) Milnor invariants become computable (isotopy-)invariants of links,
and would provide future challenges as in the perspectives mentioned above.

\section{Review: the first non-vanishing Milnor invariant of links.}\label{Srev}
To review the invariant, we begin by setting groups, which will be used throughout this paper. 
For a group $G$, we let $\Gamma_1 G$ be $G$,
and denote by $\Gamma_m G$ the commutator $ [\Gamma_{m-1} G,G]$ by induction.
Let $F$ be the free group of rank $q$.
We denote by $Q_m$ the quotient group $\Gamma_{m-1}F /\Gamma_{m}F$ with $m \geq 2 $.
Accordingly, we have the central extension,
\begin{equation}\label{kihon2} 0 \lra Q_m \lra F/\Gamma_{m}F \xrightarrow{\ \ p_{m-1}\ \ }F/\Gamma_{m-1}F \lra 1 \ \ \ \ \ (\mathrm{central \ extension})
. \end{equation}
The abelian kernel $Q_n$ is known to be free with a finite basis; see, e.g., \cite[Theorem 1.5]{CFL}.

Next, let us review the $m$-th leading terms of Milnor invariants according to \cite{Mil1,IO}.
We suppose that the reader has some knowledge of knot theory, as in \cite{Hil}.
Let us fix a link $L \subset S^3$ with $q$ components and preferred longitudes $\mathfrak{l}_{\ell} $ for $ \ell \leq q $ as elements of $ \pi_1 (S^3 \setminus L)$. 
In addition, let $f_2 : \pi_1 (S^3 \setminus L) \ra F/\Gamma_2 F=\Z^q $ be the abelianization $\mathrm{Ab}$.
Furthermore, for $m \in \mathbb{N}$, we assume:

\begin{itemize}
\item
\noindent
\textbf{Assumption $\mathcal{A}_m$}.
There are homomorphisms $f_k : \pi_1 (S^3 \setminus L) \ra F/\Gamma_k F $ for $k$ with $k \leq m$, which satisfy the commutative diagram 
$$
{\normalsize
\xymatrix{
\pi_1 (S^3 \setminus L) \ar[d]_{f_2} \ar[dr]_{f_3} \ar[drrr]_{f_4}^{} \ar[drrrrrr]^{f_m}_{\!\!\!\!\!\!\!\!\!\!\!\!\!\!\!\!\!\!\!\!\!\!\!\!\!\!\!\! \cdots \cdots }& & \\
F/\Gamma_2 F& F/\Gamma_3 F \ar[l]^{p_2 }& &F/\Gamma_4 F \ar[ll]^{p_3} &\cdots \cdots \ar[l] & & F/\Gamma_{m} F \ar[ll]^{p_{m-1} }.
}}$$
\end{itemize}

\noindent
Here, we should remark that if there is another extension $f_m' $ instead of $f_{m}$, then $f_m$ equals $f_m'$ up to conjugacy, by centrality.
Further, since the centralizer subgroup of $[x_i]$ in $ F/\Gamma_{m} F $ is known to be $\{ x_i^k\}_{k \in \Z} \times Q_m$,
the map $f_m$ sends every longitude $ \mathfrak{l}_{\ell} $ to the central subgroup $ Q_m $, up to a factor of $f_m( \mathfrak{m}_{\ell} ) ^{\pm 1 } $.
Then, the $q$-tuple
$$ \bigl( f_m( \mathfrak{l}_1 ), \dots, f_m( \mathfrak{l}_{q} ) \bigr) \in (Q_m)^q $$
is called {\it the first non-vanishing Milnor $\mu $-invariant} or {\it the $m $-th Milnor $\mu $-invariant }of $L $. 
This $\mu $-invariant is known to be a complete obstruction for lifting $f_{m}$. More precisely,
\begin{prop}[\cite{Mil2}]\label{ea211} Suppose Assumption $\mathcal{A}_m$.
Then, $f_m$ admits a lift $ f_{m+1} : \pi_1 (S^3 \setminus L) \ra F/\Gamma_{m+1} F $
if and only if all $ m $-th Milnor invariants vanish, i.e., $ f_m( \mathfrak{l}_{\ell} ) =0 \in Q_m $.
\end{prop}

In closing this section, in contrast to many studies of the Milnor $\mu $-invariant,
we should emphasize that there have been five difficulties in concretely computing the invariant:

\begin{enumerate}[(I)]
\item The quotient $F/\Gamma_{m} F$ must be quantitatively investigated.
Many papers on the Milnor invariant used the original Magnus expansion in the non-commutative power series ring $\Z\langle \! \langle X_1, \dots, X_q \rangle\! \rangle$,
which leads to a difficulty computing $ f_m( \mathfrak{l}_{\ell} )$; see \cite{Mil2,IO}.
\item The next one is complexity to present the longitude $ \mathfrak{l}_{\ell} $. 
As a solution, 
Milnor \cite[\S 3]{Mil2} (see \eqref{qq} later or \cite[Chapters 11--14]{Hil}) showed
that each $f_m (\mathfrak{l}_\ell)$ can be formulated as a word of meridians $ f_m (\mathfrak{m}_1) , \dots, f_m (\mathfrak{m}_{\# L})$
in $F/\Gamma_{m} F$, and he suggested an algorithm; however,
the algorithm becomes exponentially complicated as $m$ and $q$ increase.

\item Explicitly describing these $f_m$'s has been considered to be difficult, because of the non-commutativity of $F/\Gamma_m F$.

\item Milnor \cite{Mil2} originally defined the $\mu$-invariant with respect to a sequence $I \subset \{ 1,2,\dots, q\}^m$. 
However, the relation between the invariant and the sequence $I$ is quite complicated;
see \cite{HL}. In fact the relations given by Milnor actually form a compete and explicit set of relations. 

\item Concerning higher Milnor invariants, the definition seems rather intricate and overly algebraic (see \cite{Mil2,Hil,Stein}).
Moreover, it happens in many cases that all the higher invariants are zero (see, e.g., \cite[Table A]{Stein}), and
it is hard to check how strong the invariants are and whether they are trivial or not.
\end{enumerate}

\section{Unipotent Magnus embedding.}\label{Semb}
The key to overcoming the above difficulties is the unipotent Magnus embedding \cite{GG}, which is a faithful linear representation of $ F/\Gamma_m F $. Here, we study the embedding:
we denote by $\Omega_m$ the commutative polynomial ring $\Z[\lambda_{i}^{(j)}]$
over commuting indeterminates $ \lambda_{i}^{(j)}$ with $ i \in \{ 1,2, \dots, m-1 \}, \ j \in \{ 1, \dots, q\}$.
Let $I_m$ be the identity matrix of rank $m$, and let $E_{i,j}$ be a matrix with 1 in the $(i,j)$-th position and zeros elsewhere.
Moreover, we define
$$\Upsilon_m : F \lra GL_m(\Omega_m ) $$
as a homomorphism by setting
$$ \Upsilon_m(x_j)= \left(
\begin{array}{cccccc}
1 & \lambda_{1}^{(j)} & 0& \cdots & 0\\
0 & 1& \lambda_{2}^{(j)} & \cdots & 0\\
\vdots & \vdots & \ddots & \ddots & \vdots \\
0&0 & \cdots & 1 & \lambda_{m-1}^{(j)} \\
0& 0& \cdots & 0 & 1
\end{array}
\right).
$$
As is known \cite{GG}, $\Upsilon_m ( \Gamma_m F)=\{ I_m\}$, and the induced map $ F/\Gamma_{m} F \ra GL_m( \Omega_m) $ is injective.
Thus, we have an isomorphism $F/\Gamma_{m} F \cong \mathrm{Im}(\Upsilon_m). $
Moreover, $y \in F/\Gamma_{m} F $ lies in the center $Q_m$, if and only if $\Upsilon_m(y)$ equals $I_m + \omega E_{1,m}$ for some $\omega \in \Omega_m$;
thus it is easy to deal with the center $Q_m$, via $\Upsilon_m. $
In this paper, we call the map $\Upsilon$ {\it the unipotent Magnus embedding}.

Next, in order to describe $\Upsilon_m $ in details (Lemma \ref{lemapp1} as in a Taylor expansion), let us review the Fox derivative; see, e.g., \cite{CFL}.
Namely, for each $x_k $ with $k \in \{1,\dots, q\}$, we define a map $ \frac{\partial \ \ }{\partial x_k} : F \ra \Z [F ]$
with the following two properties:
\[ \frac{\partial x_i}{\partial x_k} = \delta_{i,k}, \ \ \ \ \ \ \ \ \ \ \frac{\partial (uv)}{\partial x_k} = \frac{\partial u }{\partial x_k} \varepsilon (v ) +u \frac{\partial v}{\partial x_k}, \ \ \ \ \ \ \mathrm{for \ all \ } u, v \in F . \]
Here, $\varepsilon$ is the augmentation $ \Z [F] \ra \Z$. Further, for $ y \in F$, we define the higher derivative
$$ \frac{\partial^{n} y}{\partial x_{i_1}\cdots \partial x_{i_{n}}} = \frac{\partial \ \ }{\partial x_{i_1}} \Bigl( \frac{\partial^{n-1} y}{\partial x_{i_2}\cdots \partial x_{i_{n}} } \Bigr) $$
by induction on $n$. For short, we often abbreviate it as $D_{i_1 \cdots i_n}(y)$. 
\begin{lem}\label{lemapp1} For any $y \in F$, the image $\Upsilon_m (y)$ is formulated as
$$ \scalebox{0.8}{$
\begin{pmatrix}
1&\displaystyle{\sum_{k_1}}\varepsilon\left(D _{k_1}(y) \right)\lambda_{1}^{(k_1)}& \displaystyle{\sum_{k_1,k_2}}\varepsilon\left(D_{k_1 k_2}(y)\right)\lambda_{1}^{(k_1)}\lambda_{2}^{(k_2)}&\cdots &\ast
&\displaystyle{\sum_{k_1,\dots, k_{m-1}}}\varepsilon\left(D_{k_1 \cdots k_{m-1}}(y)\right)\lambda_{1}^{(k_1)}\cdots \lambda_{m-1}^{(k_{m-1})}\\
0&1&\displaystyle{\sum_{k_2}}\varepsilon\left( D _{k_2}(y) \right)\lambda_{2}^{(k_2)}&\cdots &\ast&\displaystyle{\sum_{k_2,\dots, k_{m-1}}}\varepsilon\left(D_{k_2 \cdots k_{m-1}}(y)\right)\lambda_{2}^{(k_2)}\cdots \lambda_{m-1}^{(k_{m-1})}\\
0&0&1&\cdots &\ast &\displaystyle{\sum_{k_3,\dots, k_{m-1}}}\varepsilon\left(D_{k_3 \cdots k_{m-1}}(y)\right)\lambda_{3}^{(k_3)}\cdots \lambda_{m-1}^{(k_{m-1})}\\
\vdots&\vdots &\vdots &\ddots&\vdots&\vdots\\
0&0&0&\cdots0&1&\displaystyle{\sum_{k_{m-1}}}\varepsilon\left(D_{k_{m-1}}(y)\right)\lambda_{m-1}^{(k_{m-1})}\\
0&0&0&\cdots &0&1
\end{pmatrix}$}.
$$
Here, the symbols $k_s, k_{s+1}\dots, k_t$ in each sum run over the product $\{1,2,\dots, q\}^{t-s+ 1} $.
\end{lem}
\noindent
This lemma can be shown by induction on the word length of $y$ or the shuffle relation \eqref{shu}.

Furthermore, we need to consider the equality \eqref{kankei} described below.
For $b \in \mathrm{Im} (\Upsilon_{m})$, we choose a preimage $ B \in p_m^{-1} (b) \subset \mathrm{Im} (\Upsilon_{m+1})$.
Here, it is worth noting from Lemma \ref{lemapp1} that the choice is a problem of choosing in the $(1,m+1)$-entry of $B$.
By centrality, we can easily see that
\begin{equation}\label{kankei} B^{-1} A B = ( B+ \omega E_{1,m+1})^{-1} A ( B+ \omega E_{1,m+1}) \in \mathrm{Im}(\Upsilon_{m+1})
\end{equation}
for any $A \in \mathrm{Im} (\Upsilon_{m+1})$ and any $\omega \in \Omega_{m+1}.$

Finally, we should comment on the work of Murasugi \cite{Mur}.
He also considered similar unipotent matrices over $\Z$ and showed a relation to Milnor invariants modulo some integers (see also \cite[\S 12.9]{Hil}).
However, his arguments encountered the difficulties (II)--(IV) in \S \ref{Srev}.

\section{Theorem on the first non-vanishing Milnor invariant.}\label{Sthm}
We will state Theorem \ref{thm1} and Corollary \ref{thm2}, which can resolve the difficulties (II) and (III).
We were inspired to develop them by the quandle cocycle invariant \cite[\S 5]{CEGS} (see \S \ref{App2} for details).

First, let us set up some notation from knot theory. 
Choose a link diagram $D$ of $L$, and suppose Assumption $\mathcal{A}_m$.
Let $l_j$ be a simple closed curve on the boundary $\partial V_j$ of a tubular neighbourhood of $j$-th component of $L$ such that $p_j\circ l_j \circ p_j^{-1}$ represents $\mathfrak{l}_j$ $(j \in \{1,\ldots, q\})$. Here, $p_j$ is a path from the base point of $\pi_1(S^3\setminus L)$ to $\partial V_j$. As illustrated in Figure \ref{logifig}, consider the over-arcs $\alpha_1, \ \alpha_2, \dots, \alpha_{N_j }$ along the orientation of the longitude $l_j$.
We may assume that $ \alpha_1 $ is equal to the meridian $\mathfrak{m}_j $.
Let $\beta_k $ be the arc that divides $ \alpha_{k}$ and $ \alpha_{k+1}$, and $\epsilon_k \in \{ \pm 1\}$ be the sign of the crossing between $ \alpha_{k}$ and $ \beta_{k}$.
Then, via Wirtinger presentation, we can regard $f_m$ as a map $ \{ \ \mathrm{arcs \ of \ } D \ \} \ra F/\Gamma_m F $; further, the preferred longitude on the link component 
is presented by
$$ \alpha_1^{- \epsilon_1}  \beta_1^{- \epsilon_1}  \alpha_2^{- \epsilon_2}  
\beta_2^{- \epsilon_2} \cdots \beta_{N_j}^{- \epsilon_{N_j}}  \alpha_{N_j}^{- \epsilon_{N_j}}   \pi_1(S^3 \setminus L). $$


\begin{figure}[tpb]
\begin{center}
\begin{picture}(50,74)
\put(-68,25){\large \ \ \ \ $\alpha_1 $}
\put(-13,24){\large $\alpha_2 $}
\put(14,24){\large $\alpha_3 $}
\put(-73,54){\Large $l_j $}

\put(-66,37){\pc{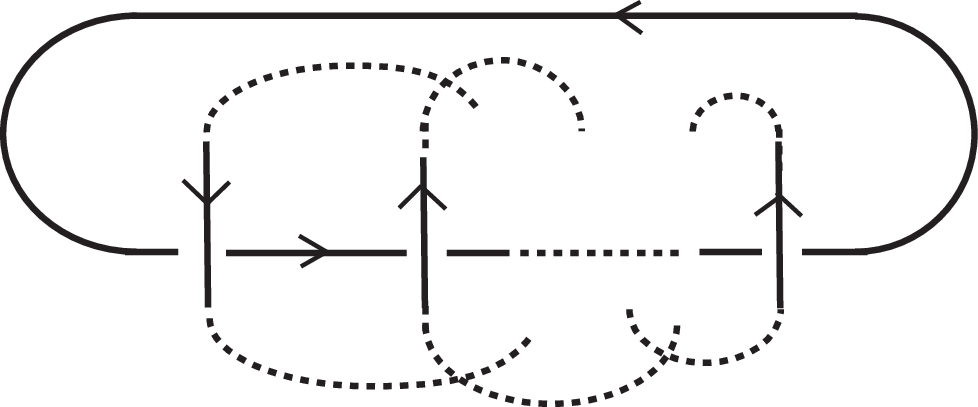}{0.34}}

\put(-39,50){\large $\beta_1 $}
\put(-6,50){\large $\beta_2 $}
\put(69,50){\large $\beta_{N_j} $}
\put(33,46){\large $\cdots $}
\end{picture}
\end{center}
\vskip -1.7pc
\caption{\label{logifig} The longitude $l_j$ and arcs $\alpha_i$ and $\beta_i$ in the diagram $D$. }
\end{figure}
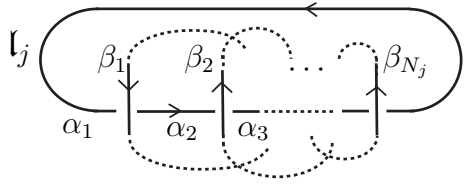

Next, we inductively define
an assignment $\mathcal{C}: \{ \alpha_k \}_{k \leq N_{j}} \ra \mathrm{Im}(\Upsilon_{m+1}) $.
Let $\mathcal{C} (\alpha_1 ) =\Upsilon_{m+1} (x_j) $.
For $1 < k \leq N_{j } $ and by choosing $B_{k} \in p_m^{-1} \bigl( f_m( \beta_k) \bigr) $, we can define $\mathcal{C} (\alpha_{k+1} )$ by
$B_k^{- \epsilon_k }\cdot \mathcal{C} (\alpha_{k} ) \cdot B_k^{\epsilon_k }$.
The independence of the choice of $B_k$ follows from \eqref{kankei}.
In summary, we should notice that this assignment canonically extends to a map $\mathcal{C}: \{ \ \mathrm{arcs \ of \ } D \ \} \ra \mathrm{Im}(\Upsilon_{m+1}) $.

Before proving Theorem \ref{thm1}, let us define the difference between the first and the $N_j$-th crossings as the following: 
\begin{equation}\label{fm+11}
\Psi_{m}(j) :=\Upsilon_{m+1} (x_j) ^{-1} \cdot B_{N_j}^{- \epsilon_{N_j} } \cdot \mathcal{C} (\alpha_{N_j} ) \cdot B_{N_j}^{ \epsilon_{N_j} } \in \Upsilon_{m+1}( Q_{m+1}).
\end{equation}
The point here is that the definition of $ \Psi_{m}(j) $ does not use any explicit description of the longitude $\mathfrak{l}_j $.
Furthermore, for any $k \in \mathbb{N}$ and $ j \leq \# L$, consider a map
\begin{equation}\label{iij} \Gamma_{k} F \lra \Gamma_{k+1} F; \ \ \ \ \ y \longmapsto x_j^{-1} y^{-1} x_j y .
\end{equation}
Then, this map with $k=m$ induces an additive homomorphism
$$ \mathcal{I}_j : \Upsilon_m ( Q_m) \lra \Upsilon_{m+1} ( Q_{m+1}).$$
Then, via $\mathcal{I}_j$, the formula $\Psi_{m}(j) $ is equivalent to the $m$-th Milnor invariant $f_m (\mathfrak{l}_j) \in Q_m $: 
\begin{thm}\label{thm1} This map $ \mathcal{I}_j$ is injective, and
the equality $ \mathcal{I}_j \bigl( \Upsilon_m \circ f_m (\mathfrak{l}_j) \bigr) = \Psi_{m} (j)$ holds.
\end{thm}
\begin{proof}
Injectivity can be easily verified from the equality \eqref{ghgh} in \S \ref{Sbra1}. 
To prove the latter assertion, let us denote $B_{1}^{ \epsilon_{1} }B_{2}^{ \epsilon_{2} } \cdots B_{N_j}^{ \epsilon_{N_j} }$ by $\mathcal{B}$
and use the abbreviation $[g,h]= ghg^{-1}h^{-1}$.
Then, $\Psi_{m}(j)= \Upsilon_m ([x_j^{-1}, \mathcal{B}^{-1}]) $ by definition.
In addition, notice the elementary equality
$[x_j,z] = [x_j, p_{m}(z)] \in Q_{m+1}$ for any $z \in \Gamma_{m-1} F$, with abuse of notation.
Since $p_{m }(\mathcal{B}) = f_m(\mathfrak{l}_j)$ from Figure \ref{logifig}, we have
$\mathcal{I}_j \bigl( \Upsilon_m \circ f_m (\mathfrak{l}_j)\bigr) = \Upsilon_m ([ x_j^{-1}, f_m(\mathfrak{l}_j^{-1})])= \Upsilon_m ([x_j^{-1}, \mathcal{B}^{-1}]) = \Psi_{m} (j) $, from which the theorem immediately follows.
\end{proof} 

As a result, we can give an explicit description of the lift $ f_{m+1} $ under an assumption. 
\begin{cor}[cf. Proposition \ref{ea211}]\label{thm2}
Suppose Assumption $\mathcal{A}_{m}$ and that all the Milnor invariants $ f_m (\mathfrak{l}_\ell) $ are zero for any $\ell \leq \# L$.
Then, the above assignment $\mathcal{C}: \{ \ \mathrm{arcs \ of \ } D \ \} \ra \Upsilon_{m+1} ( F/ \Gamma_{m+1}F) $ defines a homomorphism $f_{m+1}: \pi_1 (S^3 \setminus L) \ra \Upsilon_{m+1} ( F/ \Gamma_{m+1}F) $.
\end{cor}
\begin{proof}
From the construction of $\mathcal{C}$,
for any $j$ and any $1 \leq k < N_j$,
the Wirtinger relation of the $k$-th crossing is satisfied.
Notice $\Psi_{m} (j)=0$ from Theorem \ref{thm1};
Hence, $\mathcal{C}$ satisfies the Wirtinger presentation for every crossing, leading to the desired homomorphism.
\end{proof}

In conclusion, let us describe the procedure for computing Milnor invariants: 
Starting from the abelianization $f_2 : \pi_1 (S^3 \setminus L) \ra \Z^{q}$,
we may suppose the existence of homomorphisms $f_k : \pi_1 (S^3 \setminus L) \ra F/\Gamma_k F$ with any $k \leq h $ (i.e., Assumption $\mathcal{A}_h$),
and the $h$-th Milnor invariant $f_{h}(\mathfrak{l}_{\ell}) $ is zero.
Then, Corollary \ref{thm2} implies the existence of $f_{h+1}$ together with explicit presentations.
Iterating this process, let us suppose the existence of a minimum $m$ with $f_{m}(\mathfrak{l}_{\ell}) \neq 0$ for some $\ell $, 
that is, the assumption in Theorem \ref{thm1}, which gives the diagrammatic computation of the $m$-th Milnor invariant.
Here, since the ring $\Omega_m$ is commutative, this procedure is compatible with computer programs found in software such as Mathematica.
To sum up, these results overcome the difficulties (I)--(III). 

\section{Refinement of the higher $\bar{\mu}$-invariant.}\label{SHigher}
In the original papers \cite{Mil1,Mil2}, Milnor defined $\bar{\mu}$-invariants even in the case $f_m(\mathfrak{l}_{\ell})\neq 0,$
as mentioned in the fifth difficulty (V). 
While Habegger and Lin \cite{HL} considered a lift of Milnor invariants to string links, 
we will give a new approach which is also applicable to tame links. 
This section gives a refinement of the $\bar{\mu}$-invariants and describes a computation of the higher invariants, similar to Theorem \ref{thm1}.

First, let us explain the idea.
Milnor \cite{Mil2} (see also \cite[\S 12]{Hil}) proved that the $h$-th quotient $ \pi_1 (S^3 \setminus L) /\Gamma_h \pi_1 (S^3 \setminus L)$ has the presentation,
\begin{equation}\label{qq}
\bigl\langle \ \ x_1, \dots, x_q\ \ \bigl| \ \ [x_j, w_j]=1 \mathrm{ \ for \ } j \leq q , \ \ \Gamma_h F \ \ \bigr\rangle,
\end{equation}
where $x_j$ and $ w_j$ are represented by the $j$-th meridian and the $j$-th longitude, respectively (where $w_j$ depends on $h$).
Thus, $F/\Gamma_h F$ surjects onto this quotient group.
Since $ F/\Gamma_h F \cong \mathrm{ Im}\Upsilon_h$, there is
an isomorphism $\bar{f}_h$ from the group \eqref{qq}
to $ \mathrm{ Im}\Upsilon_h/ \mathcal{N}_{h ,L}$ for some normal subgroup $\mathcal{N}_{h ,L}$.
Furthermore, from Proposition \ref{ea211},
$ w_j$ lies in $\Gamma_{m} F $ if and only if $[x_j, w_j] $ is contained in $\Gamma_{m+1} F $.
Thus, it is reasonable to consider this $[x_j, w_j] $ (instead of $w_j$) to be an obstruction of such isomorphisms $\bar{f}_h$.
To conclude, we will define the higher invariants as $[x_j, w_j] $, i.e., as something quantitative in $ \mathrm{ Im}\Upsilon_h/ \mathcal{N}_{h ,L}$.

On the basis of this idea, we define the normal subgroups $ \mathcal{N}_{h ,L}$ and homomorphisms
$$\overline{f}_{h}: \pi_1 (S^3 \setminus L) \lra \mathrm{Im}(\Upsilon_{h})/ \mathcal{N}_{h,L} $$
which satisfy $ \overline{f}_{h-1}=
[p_{h-1}] \circ \overline{f}_{h} $ and $\overline{f}_{h}( \mathfrak{m}_j ) =[\Upsilon_h(x_j)]$ by induction on $h $.
Here, we also introduce a finite set $ K_h \subset \mathrm{Im}( \Upsilon_h) $ by induction.
Suppose the assumption $\mathcal{A}_m$.
To begin, if $h=m$, we let $ K_m $ be the empty set, $\mathcal{N}_{m,L}$ be zero, and $ \overline{f}_m=f_m$.
Next, suppose that we can define such $ \overline{f}_{h}$ and $\mathcal{N}_{h,L}$. 
Recall the projection $p_h : \mathrm{Im}( \Upsilon_{h+1}) \ra \mathrm{Im}( \Upsilon_h)$, and
choose a section $\mathfrak{s}_h : \mathrm{Im}( \Upsilon_h) \ra \mathrm{Im}( \Upsilon_{h+1}) $. 
In an analogous way to \eqref{fm+11}, taking the arcs $\alpha_k $ and $ \beta_k $ from Figure \ref{logifig}, 
we will inductively define $\mathcal{C}_{h} (\alpha_k ) $ which lies in the quotient $ \mathrm{Im}( \Upsilon_{h+1} ) / \langle \mathfrak{s}_h (K_h) \rangle.$ %
Here, $ \langle \mathfrak{s}_h (K_h) \rangle$ is the normal closure of $\mathfrak{s}_h (K_h) $.
Let $\mathcal{C}_h (\alpha_1 ) =\Upsilon_{h+1} (x_j) $.
For $1 < k \leq N_{j }, $ choosing a representative $B_{k} \in \mathrm{Im}( \Upsilon_{h+1} ) $
such that $[p_h (B_k)]= \bar{f}_h( \beta_k) $, we can define $\mathcal{C}_{h} (\alpha_{k+1} )$ by
$B_k^{- \epsilon_k }\cdot \mathcal{C}_{h} (\alpha_{k} ) \cdot B_k^{\epsilon_k }$.
Moreover, let us analogously define the difference between the first and the the $N_j$-th crossing as
\begin{equation}\label{fm+1122}
\bar{\mu}_{L}^{h} (j) :=\Upsilon_{h+1} (x_j) ^{-1} \cdot B_{N_j}^{- \epsilon_{N_j} } \cdot \mathcal{C}_{h} (\alpha_{N_j} ) \cdot B_{N_j}^{ \epsilon_{N_j} } \in \mathrm{Im}( \Upsilon_{h+1} ) / \langle \mathfrak{s}_h (K_h) \rangle.
\end{equation}
It is worth noting that $ \bar{\mu}_{L}^{h} (j)$ lies in $\Upsilon_{h+1} (Q_{h+1}) / \bigl( \langle \mathfrak{s}_h (K_h) \rangle \cap \Upsilon_{h+1} (Q_{h+1}) \bigr)$, since
$ \bar{p}_{h+1} (\bar{\mu}_{L}^{h} (j))) =\bar{\mu}_{L}^{h-1} (j)= 0$ by induction; we later show that this central quotient is independent of the choice of $ \mathfrak{s}_h $ (see Proposition \ref{oo}).
Next, we define $K_{h+1} \subset \mathrm{Im}( \Upsilon_{h+1}) $ to be the union
$$ \{ \mathfrak{s}_h (K_h) \} \cup \{ \bar{\mu}_{L}^{h} (1),\dots,\bar{\mu}_{L}^{h} (\# L) \} ,$$
and define $\mathcal{N}_{h+1,L}$ as the subgroup normally generated by $ K_{h+1}$.
To summarize, similar to Corollary \ref{thm2}, the assignment $\mathcal{C}_{h} $ defines a homomorphism
$ \overline{ f}_{h+1} : \pi_1 (S^3 \setminus L) \ra \mathrm{Im}(\Upsilon_{h+1} ) / \mathcal{N}_{h+1,L}$,
and we have the commutative diagram:
$$
{\normalsize
\xymatrix{
\pi_1 (S^3 \setminus L) \ar[d]_{f_m} \ar[dr]_{\overline{f}_{m+1}}^{ \ \cdots \cdots } \ar[drrr]_{\overline{f}_{h}} \ar[drrrrr]^{\overline{f_{h+1}}} 
& & \\
\mathrm{Im}(\Upsilon_{m} )&\mathrm{Im}(\Upsilon_{m+1} ) /\mathcal{N}_{m+1,L} \ar[l]^{\!\!\!\!\!\!\!\! p_{m} } &\cdots \cdots \ar[l] & \mathrm{Im}(\Upsilon_{h} ) / \mathcal{N}_{h,L} \ar[l]^{\!\!\!\!\!\!\!\! p_{h-1} }
& & \mathrm{Im}(\Upsilon_{h+1} )/ \mathcal{N}_{h+1 ,L} \ar[ll]^{p_{h} }.& &
}}$$

We can easily prove a theorem similar to Theorem \ref{thm1} considered modulo $\mathcal{N}_{h ,L}$.
\begin{thm}\label{thm4}For any $ j \leq \# L$,
the equality $ \bar{\mu}_{L}^{h} (j) = [ \Upsilon_{m+1}(x_j), \ \mathfrak{s}_h \bigl( \bar{f}_h (\mathfrak{l}_j ) \bigr) ] $ holds.
\end{thm}
\noindent
The proof is essentially the same as that of Theorem \ref{thm1}.
In conclusion, it is natural to define higher invariants as follows:
\begin{defn}\label{def} We define the $h$-th $\bar{\mu}$-invariant by the $\# L$-tuple
$$( \bar{\mu}_{L}^{h} (1),\dots, \bar{\mu}_{L}^{h} (\# L)) \in \bigl( \Upsilon_{h+1} (Q_{h+1}) / \bigl( \langle \mathfrak{s}_h (K_h) \rangle \cap \Upsilon_{h+1} (Q_{h+1}) \bigr) \bigr)^{\# L}. $$
\end{defn}
Here, we should show that, this definition is essentially independent of the choice of the sections $\mathfrak{s}_h $, and
that our extension is universal in some sense. To be precise,
\begin{thm}\label{thm5} The map $\bar{f}_h$ induces the group isomorphism
$$\pi_1 (S^3 \setminus L) /\Gamma_h \pi_1 (S^3 \setminus L) \cong \mathrm{Im} \Upsilon_h / \mathcal{N}_{h ,L} .$$
\end{thm}
\begin{proof}
This claim has been made before for $h=m$: see \cite{Mil2} or \cite[Theorem 12.3]{Hil}.
If the claim is true for an $h$, the map $\bar{f}_{h+1}$ modulo $\Gamma_{h+1}$ is a homomorphism between central extensions over $\mathrm{Im} \Upsilon_h / \mathcal{N}_{h ,L}$.
Since $\bar{f}_{h+1} ( [x_j, w_j]) = \bar{\mu}_{L}^{h} (j) $ by Theorem \ref{thm4}, 
the quotient of $\bar{f}_{h+1}$ is an isomorphism.
\end{proof}
Incidentally, Theorems \ref{thm4} and \ref{thm5} imply that we can recover the longitude $ \bar{f}_h (\mathfrak{l}_j ) $ from $ \bar{\mu}_{L}^{h} (j) $ modulo $\mathcal{N}_{h ,L} $.
That is, information of $ \bar{f}_h (\mathfrak{l}_j ) $ and $ \bar{\mu}_{L}^{h} (j) $ are essentially equivalent.

Finally, we mention the original $\bar{\mu}$-invariants \cite{Mil1,Mil2}.
The original ones are seemingly far from universality as in the above lifting properties.
Moreover, they are considered by passage of certain ideals ``$\Delta(I)$" of the integer ring $\Z$ (instead of $\mathcal{N}_{h,L} $), in terms of a non-commutative ring.
Thus, the passage algebraically seems stronger than the subgroup $ \mathcal{N}_{h ,L}$, and
it is hard to check where the invariants are trivial or not.
In contrast, in \S \ref{Sb31}, we give some computation of our $\bar{\mu}$-invariants with non-triviality.


\section{Examples; links with crossing number $<8 $. }\label{Sexa}
To show that our computation is faster and more manageable than the previous ones, 
we will compute the (higher) Milnor invariants of some links in terms of $\Psi_m(j)$ and $\bar{\mu}_h(j)$.
In this section, we let $\mathrm{c}(L)$ denote the (minimal) crossing number of $L$, and let $\mathrm{lk}(L) \in \Z$ be the linking number if $q=2$. 

\subsection{Bracket and standard commutators. }\label{Sbra1}

To overcome difficulty (IV) and simply describe our computation,
we introduce a bracket in $\Omega_m$.
For $r,s \in \mathbb{N}$, let $\iota_{s}:\Omega_r \hookrightarrow \Omega_{r+s }$ be the canonical inclusion,
and $\kappa_{r}:\Omega_s \hookrightarrow \Omega_{r+s }$ be
the ring homomorphism defined by $\kappa_{r}(\lambda_{i}^{(j)})= \lambda_{i+r}^{(j)}$. Then, we can define a bilinear map
$$ [\bullet , \bullet ]:\Omega_r \times \kappa_{r} (\Omega_s) \lra \Omega_{r+s}; \ \ \ \ (a,\kappa_{r}(b)) \longmapsto \iota_s(a) \kappa_r(b)- \iota_r(b) \kappa_s (a).$$
As mentioned in \S 3, the center $ \Upsilon_s(Q_s)$ can be regarded as a submodule of $\Omega_s$, and this bilinear map descends to 
$$ [\bullet , \bullet ]: \Upsilon_r(Q_r) \times \kappa_{r} \bigl(\Upsilon_s(Q_s)\bigr) \lra \Upsilon_{r+s}(Q_{r+s}).$$
This bracket can be interpreted as the image of the commutator. More precisely, we have
\begin{equation}\label{ghgh}
\Upsilon_{r+s}(ghg^{-1}h^{-1}) = [\Upsilon_r(g) , \ \kappa_{r}\bigl(\Upsilon_s(h)\bigr) ] , \end{equation}
for any $g \in \Gamma_{r-1}F$ and $h \in \Gamma_{s-1}F$.
This equality can be easily shown by direct computation of $\Upsilon_{r+s}(ghg^{-1}h^{-1}) $ as upper triangular matrices.

We should mention Corollaries 2.2--2.3 in \cite{CFL}, which show that
certain (Jacobi) relations ``(S1), (S2), (S2$^{\circ }$), (S3)" on the (standard) commutators in $F/\Gamma_mF$ characterize a basis of $Q_m$.
Thus, using the formula \eqref{ghgh}, we can obtain similar relations for the bracket, and it is reasonable to express Milnor invariants in terms of our brackets (see Table \ref{table1}).

Furthermore, let us consider a simple example, i.e., the left collecting commutator.
Let $\mathfrak{S}_2$ be the permutation group on $2$ elements.
For a multi-index $J=(j_1\cdots j_n) \in \{1,2, \dots, q \}^{n}$ and $\sigma=(\sigma_1,\ldots, \sigma_{n-1}) \in (\mathfrak{S}_2)^{n-1}$,
we define $\sigma(J)=(j_1^{\sigma},\dots ,j_n^{\sigma}) \in \mathbb{N} ^n $ by
$$
(j_1^{\sigma},\dots, j_n^{\sigma})=\sigma_{n-1}(\sigma_{n-2}(\cdots\sigma_4(\sigma_3(\sigma_2(\sigma_1(j_1, j_2),j_3),j_4),j_5)\cdots),j_{n}).
$$
Then, the left collecting commutator is, by definition, formulated as follows:
$$
[[\cdots[[\lambda_1^{(j_1)}, \lambda_2^{(j_2)}],\lambda^{(j_3)}_3]\cdots],\lambda^{(j_n)}_n]=
\sum_{\sigma \in (\mathfrak{S}_2)^{n-1}} \mathrm{sign}(\sigma)\cdot \lambda_{1}^{(j_1^{\sigma})}\lambda_{2}^{(j_2^{\sigma})}\cdots \lambda_{n}^{(j_n^{\sigma})} \in \Omega_{n+1 }.
$$
For example, when $n=3$, the bracket is written as
\[ [[\lambda_1^{(i)}, \lambda_2^{(j)}], \lambda_3^{(k)}]= \lambda_{1}^{(i)} \lambda_{2}^{(j)} \lambda_{3}^{(k)}- \lambda_{2 }^{(i)} \lambda_{1}^{(j)} \lambda_{3}^{(k)} + \lambda_{3}^{(i)} \lambda_{2}^{(j)} \lambda_{1}^{(k)}- \lambda_{2 }^{(i)} \lambda_{3}^{(j)}
\lambda_{1}^{(k)} \in \Omega_4 .\]
In addition, formula \eqref{ghgh} inductively implies
\begin{equation}\label{lem9} \Upsilon_m( [[[\cdots[[x_{j_1},x_{j_2}],x_{j_3}]\cdots],x_{j_{m-2}}],x_{j_{m-1}}])= [[[\cdots[[ \lambda_{1}^{(j_1)}, \lambda_2^{(j_2)}], \lambda_3^{(j_3)}] \cdots], \lambda_{m-2}^{(j_{m-2})}], \lambda_{m-1}^{(j_{m-1})}] . \end{equation}

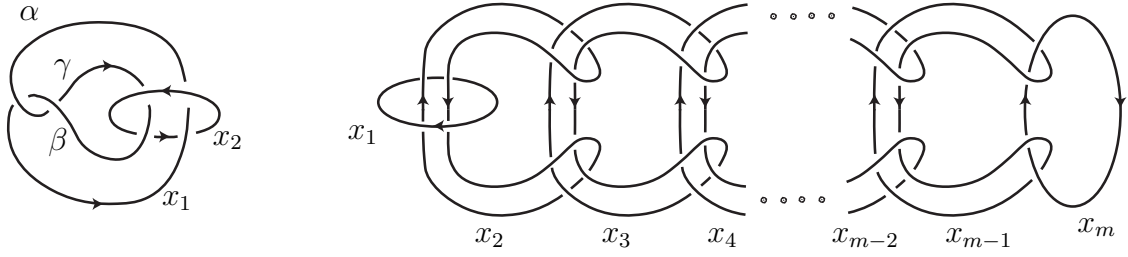
\begin{figure}[tpb]
\begin{center}
\begin{picture}(50,99)
\put(-138,15){\large \ \ \ \ $x_1 $}
\put(-103,38){\large $x_2 $}
\put(-181,36){\large \ \ \ \ $\beta$}
\put(-179,65){\large \ \ \ \ $ \gamma $}
\put(-192,85){\large \ \ \ \ $\alpha $}

\put(-52,40){\large $x_1 $}
\put(-4,1){\large $x_2 $}
\put(44,1){\large $x_3 $}
\put(84,1){\large $x_4$}

\put(131,1){\large $x_{m-2}$}
\put(174,1){\large $x_{m-1}$}
\put(224,5){\large $x_{m}$}

\put(-186,45){\pc{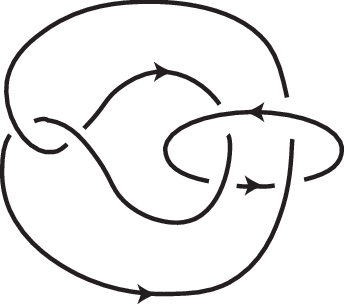}{0.4884}}

\put(-46,47){\pc{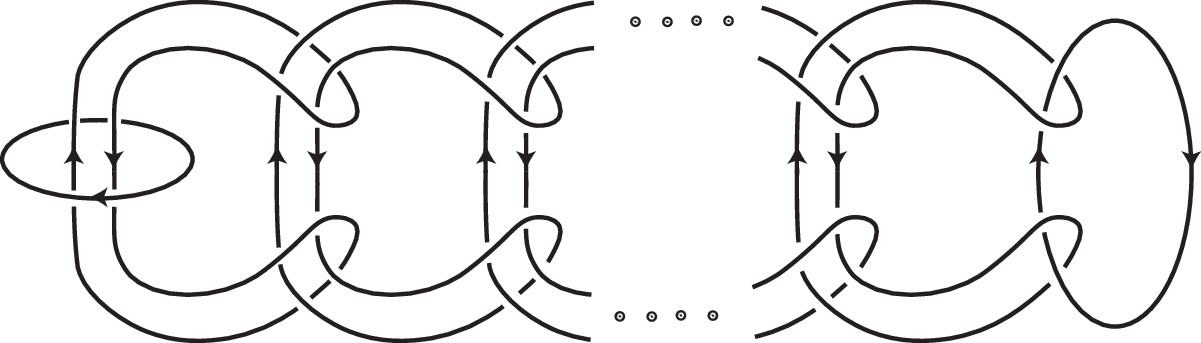}{0.49}}

\end{picture}
\end{center}
\vskip -1.1pc
\caption{\label{bb} Whitehead link and Milnor link of length $m$. }
\end{figure}
\noindent

\subsection{Whitehead link and small links. }\label{Sexample1}
Returning to the link invariants, we will first focus on the Whitehead link (cf. \cite[\S 10.3]{IO}, \cite[\S 8]{Mur} and \cite{Stein} as a known computation).
For this, fix meridians $x_1,x_2$ and three arcs $\alpha,\beta,\gamma$, as shown in Figure \ref{bb}.
Since all the invariants of degree $<4$ are known to be zero, we let $m=4$.
By using the Wirtinger presentation, the assignment $f : \{ \mathrm{arcs \ of \ }D\}\ra GL_5(\Omega_5) $ imposes the equations
$$ \mathcal{C}( \alpha)= \mathcal{C}(x_2)\mathcal{C}(x_1)\mathcal{C}(x_2)^{-1}, \ \ \ \mathcal{C}( \beta)= \mathcal{C}(\alpha)^{-1}\mathcal{C}(x_1)\mathcal{C}(\alpha), \ \ \ \mathcal{C}( \gamma)= \mathcal{C}(\beta)^{-1}\mathcal{C}(\alpha)\mathcal{C}(\beta ). $$
Accordingly, the assignment $\mathcal{C}(x_i)= \Upsilon_4(x_i)$ gives the following presentations:
$$ \!\!\!\!\!\!\!\! \mathcal{C}(\alpha)= {\small \left(
\begin{array}{ccccc}
1 & \lambda_{1}^{(1)} \!\!\!\! & [\lambda_{1}^{(2)}, \lambda_{2}^{(1)}] \!\!\!\! & [\lambda_{1}^{(1)}, \lambda_{2}^{(2)}] \lambda_{3}^{(2)}\!\!\!\! & [\lambda_{1}^{(1)}, \lambda_{2}^{(2)}] \lambda_{3}^{(2)}\lambda_{4}^{(2)} \\
0 & 1& \lambda_{2}^{(1)}& [\lambda_{2}^{(2)}, \lambda_{3}^{(1)}] & [\lambda_{2}^{(1)}, \lambda_{3}^{(2)}] \lambda_{4}^{(2)} \\
0 & 0& 1& \lambda_{3}^{(1)}& [\lambda_{3}^{(2)}, \lambda_{4}^{(1)}] \\
0 & 0& 0& 1 & \lambda_{4}^{(1)} \\
0 & 0& 0& 0 & 1
\end{array}
\right)},$$
$$\mathcal{C}(\beta)= {\small \left(
\begin{array}{ccccc}
1 & \lambda_{1}^{(1)} \!\! & 0 &[[\lambda_{1}^{(1)},\lambda_{2}^{(2)} ], \lambda_{3}^{(1)}] & [[\lambda_{2}^{(1)},\lambda_{3}^{(2)} ], \lambda_{4}^{(1)}] \lambda_{1}^{(1)}-[[\lambda_{1}^{(1)},\lambda_{2}^{(2)} ], \lambda_{3}^{(1)}] \lambda_{4}^{(2)} -[\lambda_{1}^{(1)},\lambda_{2}^{(2)} ][\lambda_{3}^{(1)},\lambda_{4}^{(2)} ] \\
0 & 1& \lambda_{2}^{(1)} \!\!\!\! & 0& [[\lambda_{2}^{(1)},\lambda_{3}^{(2)} ], \lambda_{4}^{(1)}]\\
0 & 0& 1& \lambda_{3}^{(1)}& 0 \\
0 & 0& 0& 1& \lambda_{4}^{(1)} \\
0 & 0& 0& 0 & 1
\end{array}
\right)}.
$$
Furthermore, the products $\Phi_{5}( j)$ in \eqref{fm+1} are, by definition, formulated as
\[ \Psi_{4}( 1) =\Upsilon_5(x_1)^{-1} \mathcal{C}( \alpha) \mathcal{C}( x_2)^{-1}\mathcal{C}( \beta)\mathcal{C}( x_2 ) \Upsilon_5(x_1)\mathcal{C}(x_2)^{-1} \mathcal{C}( \beta)^{-1} \mathcal{C}( x_2) \mathcal{C}( \alpha)^{-1}, \]
\[ \Psi_{4}( 2) =\Upsilon_5(x_2)^{-1} \mathcal{C}( x_1 )\mathcal{C}( \beta)^{-1} \Upsilon_5(x_2) \mathcal{C}( \beta) \mathcal{C}( x_1)^{-1} \in \Upsilon_5( Q_5). \]
An elementary computation (with the help of Mathematica) leads to the conclusion,
\[ \Psi_{4}( 1)_{(1,5)}= -\Psi_{4}(2)_{(1,5)}= - [[[\lambda_{1}^{(1)},\lambda_{2}^{(2)}],\lambda_{3}^{(1)}],\lambda_{4}^{(2)} ].  \]

Here, we compare the computations of $\Psi_{4}(1)$ and $\Psi_{4}(2)$ with the original Milnor invariant of the Whitehead link. 
As seen in \cite{Mil2,Mur,Stein}, the original one is known to be  
$$
\mu(ijkl) = \begin{cases}
    1 & \text{if}\  (ijkl)=(1122) \ \text{and its cyclic permutations,} \\
    -2 &  \text{if}\  (ijkl)=(1212)\  \text{and its cyclic permutations,} \\
    0 & \text{otherwise}.
    \end{cases}
$$
Here, $\mu(ijkl)$ is originally defined by the coefficient of $\lambda_1^{(i)}\lambda_2^{(j)} \lambda_3^{(k)}  $ of $\Upsilon_3 \circ f_{3}(\mathfrak{l}_l) $
\footnote{Strictly speaking, the original definition of Milnor invariants is used by the Magnus expansion instead of the unipotent Magnus expansion. However, we can easily see that our description and the original one are equivalent from the discussion in Appendix \ref{App}.}.
Thus, using the map $\mathcal{I}_j$ in Theorem \ref{thm1}, we can 
see that 
the above computation coincides with the original one. 

Similarly, 
we can easily compute (with a computer program) the first non-vanishing Milnor invariants of other links with $n$ crossings, where $m<11$ and $n<20$.
For example, for every 2-component link $L$ with $\mathrm{lk}(L)=0$ and $\mathrm{c}(L) <9$, 
we can check a list of the $m$-th Milnor invariants; see Table \ref{table1}, where we have used the abbreviation of the brackets:
$$\varUpsilon := [[[\lambda^{(1)}_1,\lambda^{(2)}_2],\lambda^{(1)}_3],\lambda^{(2)}_4] \in \Omega_5, \ \ \ \
\varLambda := [[[[[\lambda^{(1)}_1,\lambda^{(2)}_2],\lambda^{(1)}_3],\lambda^{(1)}_4],\lambda^{(1)}_5],\lambda^{(2)}_6] \in \Omega_7.$$
Although Stein gave such a list \cite[Table A1 ]{Stein},
we should point out that his computation of the link $ 8_{10}^2$ was incorrect.

Although we computed the $m$-th invariant of every 2-component link $L$ with $\mathrm{lk}(L)=0$ and $\mathrm{c}(L) =9$,
the invariants are constant multiples of $\varUpsilon$. 
To conclude, the $m$-th Milnor link invariants are not so strong for links with $\mathrm{lk}(L)=0$ and $\# L=2.$

\begin{table}[htb]
$$ \begin{tabular}{lcrrrrrrr}\hline
Link & $5_1^2$ & $7_4^2$ & $7_6^2$ & $7_8^2$ & $8_{10}^2$ & $8_{12}^2$ & $8_{13}^2$ & $8_{15}^2$ \\ \hline \hline
$m$ & $4$ & $4$ & $4$ & $4$ & $6$ & $6$ & $4$ & $4$\\ \hline
$\Psi_{m}(1)$ & $ \varUpsilon $ & $ 2 \varUpsilon $ & $ \varUpsilon $ & $ \varUpsilon $& $ \varLambda$ & $ \varLambda$ & $ \varUpsilon $ & $ \varUpsilon $\\ \hline
$\Psi_{m}(2)$ & $ -\varUpsilon $ & $ -2 \varUpsilon $ & $- \varUpsilon $ & $- \varUpsilon $& $ - \varLambda$ & $ - \varLambda$ & $ -\varUpsilon $ & $ -\varUpsilon $\\ \hline
\end{tabular}$$
\caption{\label{table1}
The first non-vanishing Milnor invariants of links with linking number zero. } 
\end{table}

\subsection{Examples of higher $\bar{\mu}$-invariants}\label{Sb31}
We will diagrammatically compute the $\bar{\mu}$-invariants in the same way as \S 4.

Before the computation, we should analyze the generators of the group $\Upsilon_{h+1} (Q_{h+1}) / \bigl( \langle \mathfrak{s}_h (K_h) \rangle \cap \Upsilon_{h+1} (Q_{h+1}) \bigr) $ in \eqref{fm+1122}; see Proposition \ref{oo} below.
Recalling the bracket in \S \ref{Sbra1}, we will define a subgroup $\Delta_{h}$ of $\Upsilon_{h+1}(Q_{h+1})$
by induction: First, let $\Delta_m$ be zero.
Define $ \Delta_{m+1} $ as the abelian group generated by $\bigl\{ \ [\Psi_m (\ell), \Upsilon_{2}(x_j) ]\ | \ \ell \leq q, j \leq q \ \bigr\}.$
Next, if we know $\Delta_m,\dots, \Delta_h $, we can define $\Delta_{h+1}$ to be the abelian group generated by the following set:
$$ \bigl\{ [d_k, \eta ] \ | \ k \leq h , \ d _k \in \Delta_{k-1}, \ \eta \in \Upsilon_{m-k+1} (Q_{m-k+1}) \ \bigr\} \cup \bigl\{ \ [\bar{\mu}_L^{h-1}(\ell), \Upsilon_{2}(x_j) ]\ | \ \ell \leq q, j \leq q \ \bigr\} . $$
From the triangularity of $\mathrm{Im}\Upsilon_h$, the following proposition can be easily shown from 
the inductive construction of $\mathcal{N}_{h,L}$.
Recall that $\bar{\mu}_L^h(j)$ defined in \eqref{fm+1122} lies in the group
$\Upsilon_{h+1} (Q_{h+1}) / \bigl( \langle \mathfrak{s}_h (K_h) \rangle \cap \Upsilon_{h+1} (Q_{h+1}) \bigr) $. 

\begin{prop}\label{oo} The group $\langle \mathfrak{s}_h (K_h) \rangle \cap \Upsilon_{h+1} (Q_{h+1}) $ 
is equal to $\Delta_{h}.$
\end{prop}
As a result, if we concretely describe the finite generators of $\Delta_{h} $, it is not so hard to 
check for the non-triviality of the higher $\bar{\mu}$-invariant. 

Some examples are shown below without any detailed proof:
\begin{exa}\label{exa3}
Concerning the links $5_1^2 , \ 7_6^2, \ 7_8^2, \ 8_{13}^2 $, the first non-vanishing invariants with $m=4$ are equal (see Table \ref{table1}).
We will briefly mention the higher invariants.
When $h=5$, 
our $\bar{\mu}_L^5(j)$ invariants are unfortunately equal to $ \pm [[[[\lambda_1^{(2)}, \lambda_2^{(1)}],\lambda_3^{(1)}],\lambda_4^{(1)}],\lambda_5^{(2)}]$
(it is worth pointing out that the original $\bar{\mu}$-invariants of length $5$ are zero).
In addition, if $h=6,7$, the higher invariants of degree $h$ are zero. 
\end{exa}
In contrast, our experience has shown that the higher invariants are useful for links with $\# L\geq 3$, and for links of $\# L=2$ with $ \mathrm{lk}(L)\geq 2$.
We give some examples below.
\begin{exa}\label{exa2}Let $L$ be the Borromean rings $6^3_2$. Since the first non-vanishing invariant $\bar{\mu}_L^{3}(j) $ is
$[[\lambda_1^{(j)}, \lambda_2^{(j+1)} ], \lambda_3^{(j+2)}]$ with $j \in \Z/3$ by Theorem \ref{thm7}, the group $\Delta_4$ reduces to
$$\Delta_4 = \Z \langle[[[\lambda_1^{(j)}, \lambda_2^{(j+1)} ], \lambda_3^{(j+2)}], \lambda_4^{(k)}] \rangle_{j,k \in \Z/3 }.$$
Furthermore, if $h = 4$, we obtain the resulting computation:
$$ \bar{\mu}_L^{4}(j) \equiv [[[\lambda_1^{(j)}, \lambda_2^{(j+1)} ], \lambda_3^{(j+1)}], \lambda_4^{(j+2)}] \ \ \ \mathrm{modulo} \ \Delta_4. $$
Hence, if $h = 5$, the group $\Delta_5$ is spanned by
$$\langle \ [[ [[\lambda_1^{(j)}, \lambda_2^{(j+1)} ], \lambda_3^{(j+2)}], \lambda_4^{(k)}],\lambda_5^{(\ell)} ] , \
[[[\lambda_1^{(j)}, \lambda_2^{(j+1)} ], \lambda_3^{(j+2)}], [\lambda_4^{(k)},\lambda_5^{(\ell)}] ], \
[\bar{\mu}_L^{4}(j) ,\lambda_5^{(k)} ] \ \rangle_{j,k, \ell \in \Z/3 }.$$
Furthermore, if $h = 5$, we can similarly obtain the computation:
$$ \bar{\mu}_L^{5}(j) \equiv [[[[\lambda_1^{(j)}, \lambda_2^{(j+1)} ], \lambda_3^{(j+1)}], \lambda_4^{(j+1)}],\lambda_5^{(j+2)} ] \ \ \ \mathrm{modulo} \ \Delta_5. $$
Then, from the description of $\Delta_k$ with $k = 4,5$, these $\bar{\mu}_L^{k+1}(j)$ are not zero modulo $\Delta_{k+1}$.
Furthermore, in our experience, $\bar{\mu}_L^{h}(j)$ becomes more complicated as $h$ increases.

We can verify, by computation, that
the links with $\mathrm{c}(L)<11$ whose first non-vanishing invariant is the same as that of the Borromean rings are only the links $L'=9_{n25}^3$ and $L''=10_{a151}^3$.
In addition, the 4-th invariants can be computed as
\begin{eqnarray*}
\bar{\mu}_{L'}^{4}(1) & \equiv & [[[\lambda_1^{(2)}, \lambda_2^{(3)} ], \lambda_3^{(3)}], \lambda_4^{(1)}] -[[[\lambda_1^{(3)}, \lambda_2^{(2)} ], \lambda_3^{(2)}], \lambda_4^{(1)}]-[[[\lambda_1^{(3)}, \lambda_2^{(1)} ], \lambda_3^{(1)}], \lambda_4^{(3)}],\\
\bar{\mu}_{L'}^{4}(2) &\equiv & [[[\lambda_1^{(2)}, \lambda_2^{(3)} ], \lambda_3^{(3)}], \lambda_4^{(1)}], \\
\bar{\mu}_{L'}^{4}(3) & \equiv & [[[\lambda_1^{(3)}, \lambda_2^{(1)} ], \lambda_3^{(1)}], \lambda_4^{(3)}] ,\\
\bar{\mu}_{L''}^{4}(1) & \equiv & [[[\lambda_1^{(2)}, \lambda_2^{(3)} ], \lambda_3^{(3)}], \lambda_4^{(1)}] -[[[\lambda_1^{(1)}, \lambda_2^{(2)} ], \lambda_3^{(2)}], \lambda_4^{(1)}],\\
\bar{\mu}_{L''}^{4}(2) &\equiv & [[[\lambda_1^{(3)}, \lambda_2^{(1)} ], \lambda_3^{(1)}], \lambda_4^{(1)} +\lambda_4^{(2)}]- [[[\lambda_1^{(3)}, \lambda_2^{(2)} ], \lambda_3^{(2)}], \lambda_4^{(3)}],\\
\bar{\mu}_{L''}^{4}(3) & \equiv& [[[\lambda_1^{(1)}, \lambda_2^{(2)} ], \lambda_3^{(2)}], \lambda_4^{(3)}].
\end{eqnarray*}
Thus, we may hope that the higher invariants are strong for $\# L \geq 3$.
Furthermore, it is interesting that the multivariable Alexander polynomials $\Delta_L$ of $L=6^3_2$ and $L'=9_{n25}^3$ are equal (cf. \cite[Theorems 4.1--4.3]{Mur2} which discussed the relation between $\Delta_L$ and $\mu$-invariants).

We also remark that the Borromean rings is a special case of $n$-th Milnor link with $n=3$. For $n$-th Milnor link, we can determine the all the first non-vanishing Milnor invariant (c.f. Appendix \ref{App3}).
\end{exa}
\begin{exa}\label{exa3}
Next, we will focus on the case $ \# L =2$.
The previous papers on Milnor invariants, e.g., \cite{IO, Stein, Hil}, mainly considered links with $ \mathrm{lk}(L) =0$.
\begin{lem}\label{lem99}
Assume $ \# L =2$.
Every higher $\bar{\mu}$-invariant of $L$ is annihilated by $ \mathrm{lk}(L) \in \Z. $
\end{lem}
\begin{proof}
Since the first non-vanishing invariant forms $ \mathrm{lk}(L) [\lambda_1^{(1)},\lambda_2^{(2)}]$,
Proposition \ref{oo} implies that the group $\Delta_h$ with $h>2$ is annihilated by $ \mathrm{lk}(L) \in \Z $;
so are the higher $\bar{\mu}$-invariants. \end{proof}
However, we hope that the higher invariants in the case $ \mathrm{lk}(L)\geq 3$ are powerful.
For example, let us focus on the links with $\# L=2$ with $ \mathrm{lk}(L)=3$.
The table below is a list of all the links with $\mathrm{c}(L)\leq 9$ and $ \mathrm{lk}(L)=3$
and of the associated higher $\bar{\mu}$-invariants.
To conclude, we can verify from the delta $\Delta_5$ in this table that the higher $\bar{\mu}$-invariants mutually detect the links.
\begin{table}[htb]
$$ \begin{tabular}{lcrrrr}\hline
Link $L$ & $\bar{\mu}_{L}^3(j)$ & $\bar{\mu}_{L}^4(j)$ & $\Delta_5 $ & $\bar{\mu}_{L}^5(1)$ & $\bar{\mu}_{L}^5(2)$ \\ \hline \hline
$6_1^2$ & 0 & $2 b_1+b_2+ b_3 $ & $\langle 3, B+D-F, A+C +E \rangle $ & $ -A-C $ & $ -A+C+D $ \\ \hline
$6_2^2$ & 0 & $2 b_1+b_2 $ & $\langle 3, B- E, A-C \rangle $ & $ D-F $ & $ B-F$ \\ \hline
$8_{a10}^2$ & 0 & $2 b_1+b_2 -b_3$ & $\langle3, B-D+F, A-C -E \rangle $& $C$ & $ -C$ \\ \hline
$8_{a11}^2$ & 0 & $2 b_1-b_2 +b_3$ & $\langle3, B+D-F, A+C +E \rangle $ & $-C$ & $ C$ \\ \hline
$9_{a23}^2$ & 0 & $2 b_1+b_2 +b_3$ & $\langle3, B+D-F, A+C +E \rangle $& $-A+ B+D$ & $ -A+C-D$ \\ \hline
$9_{a28}^2$ & 0 & $2 b_1 +b_2$ & $\langle 3, B- E, A-C \rangle $ & $-D-F$ & $ -B-D-F$ \\ \hline
$9_{a32}^2$ & 0 & $2 b_1+b_2 +b_3$ & $\langle3, B+D-F, A+C +E \rangle $& $-A+C$ & $A-C $ \\ \hline
$9_{a33}^2$ & 0 & $2 b_1+b_2 -b_3$ & $\langle3, B-D+F, A-C -E \rangle $ & $ -A+B+C$ & $ -A+B+C $ \\ \hline
$9_{n15}^2$ & 0 & $2 b_1+b_2 +b_3$ & $\langle3, B+D-F, A+C +E \rangle $ & $ -A-C+D$ & $ -A+C-D $ \\ \hline
$9_{n16}^2$ & 0 & $2 b_1+b_2 +b_3$ & $\langle3, B+D-F, A+C +E \rangle $ & $ -A-C+D$ & $ B-C+D $ \\ \hline
\end{tabular} $$
\end{table}

Here, $j \in \{ 1,2\}$, and the following formulas define the symbols:
$$b_1=[[[\lambda_{1}^{(1)},\lambda_{2}^{(2)}],\lambda_{3}^{(2)}],\lambda_{4}^{(2)}] , \ \ b_2=[[[\lambda_{1}^{(2)},\lambda_{2}^{(1)}],\lambda_{3}^{(1)}],\lambda_{4}^{(1)}] , \ \ b_3=[[[\lambda_{1}^{(1)},\lambda_{2}^{(2)}],\lambda_{3}^{(2)}],\lambda_{4}^{(1)}],$$
$$A= [[[[\lambda_{1}^{(1)},\lambda_{2}^{(2)}],\lambda_{3}^{(2)}],\lambda_{4}^{(2)}] ,\lambda_{5}^{(2)}], \ \ \ \ \ \ \ \
B= [[[[\lambda_{1}^{(1)},\lambda_{2}^{(2)}],\lambda_{3}^{(2)}],\lambda_{4}^{(2)}] ,\lambda_{5}^{(1)}],$$
$$C= [[[[\lambda_{1}^{(2)},\lambda_{2}^{(1)}],\lambda_{3}^{(1)}],\lambda_{4}^{(1)}] ,\lambda_{5}^{(2)}], \ \ \ \ \ \ \ \
D= [[[[\lambda_{1}^{(2)},\lambda_{2}^{(1)}],\lambda_{3}^{(1)}],\lambda_{4}^{(1)}] ,\lambda_{5}^{(1)}],$$
$$E= [[[[\lambda_{1}^{(2)},\lambda_{2}^{(1)}],\lambda_{3}^{(1)}],\lambda_{4}^{(2)}] ,\lambda_{5}^{(2)}], \ \ \ \ \ \ \ \
F= [[[[\lambda_{1}^{(1)},\lambda_{2}^{(2)}],\lambda_{3}^{(2)}],\lambda_{4}^{(1)}] ,\lambda_{5}^{(1)}] .$$
\end{exa}

\subsection*{Acknowledgments}
The authors wish to express their gratitude to Professors Akira Yasuhara and Kazuo Habiro for their valuable comments. They also gratefully acknowledge
valuable comments provided by the anonymous referee.

\appendix \section{Appendix: the original Magnus expansion.}\label{App}
When studying the $\bar{\mu}$-invariant, one often uses the Magnus expansion (see \cite{Mil1,IO,Hil}).
Thus, we will describe the relation between the expansion and the unipotent one $\Upsilon_m$. Here, we should recall the definition of $\Upsilon_m $ in \S \ref{Semb}. 

First we will mention the shuffle relation \eqref{shu}.
As in \cite[\S 2]{CFL}, a sequence $(c_1c_2 \cdots c_k) \in \{ 1,\dots ,n\}^k$ is called {\it the resulting shuffle of two sequences} $I=a_1 a_2 \cdots a_{|I|} $ and $J=b_1b_2 \cdots b_{|J|}$
if there are $|I|$ indexes $\alpha (1), \alpha(2), \cdots ,\alpha(|I|)$ and $|J|$ indexes $\beta (1), \beta (2), \dots, \beta (|J|)$ such that
\begin{enumerate}[(i)]
\item $ 1 \leq \alpha (1) < \alpha (2)< \cdots < \alpha (|I|) \leq k,$ \ and $ \
\ 1 \leq \beta (1) < \beta (2)< \cdots < \beta (|J|) \leq k$.
\item $c_{\alpha(i)}= a_i$ and $c_{\beta(j)}= b_j$ for some $i \in \{ 1,2, \dots, |I|\}$ \ $j \in \{ 1,2, \dots, |J|\} $.
\item each index $s \in \{ 1, 2, \dots ,k \}$ is either an $\alpha (i)$
for some $i$ or a $\beta(j)$ for some $j$ or both.
\end{enumerate}
Let the symbol $\mathrm{Sh} (I,J)$ denote the set of resulting shuffles of $I$ and $J$.
Then, \cite[Lemma 3.3]{CFL} shows the following shuffle relation, for any multi-indexes $I$ and $J$ and $y \in F$: 
\begin{equation}\label{shu}\varepsilon (D_{I}(y)) \cdot \varepsilon (D_{J}(y)) =\sum_{K \in \mathrm{Sh}(I,J)} \varepsilon (D_{K}(y)) \in \Z.
\end{equation}

In addition, we will review the Magnus expansion modulo degree $m$.
Let $\Z \langle X_1,\dots, X_q\rangle $ be the polynomial ring with non-commutative indeterminates $X_1,\dots, X_q $, and $\mathcal{J}_{m}$ be the two-sided ideal generated by polynomials of degree $\geq m $.
Then, {\it the Magnus expansion (of the free group $F$)} is the map
$ \mathcal{M}: F \ra \Z \langle X_1,\dots, X_q \rangle /\mathcal{J}_{m} $
defined by
\begin{equation}\label{mag} \mathcal{M}(y) = \varepsilon (y)+ \sum_{ n =1}^m \ \sum_{ (i_1, \dots , i_n) \in \{ 1, 2, \dots, q \}^n }\varepsilon (D_{ i_1 \cdots i_n }(y))\cdot X_{i_1} X_{i_2}\cdots X_{i_n} .
\end{equation}
As is known, this $\mathcal{M}$ is a homomorphism, and $\mathcal{M} (\Gamma_m F ) =0 $. By passage to this $\Gamma_m F$, it further induces
an injective homomorphism
$$ \mathcal{M}:F/ \Gamma_m F \lra \Z \langle X_1,\dots, X_q \rangle /\mathcal{J}_{m}. $$
Moreover, it follows from \cite[Theorem 3.9]{CFL} that the image is completely characterized by
\begin{equation}\label{image} \Bigl\{ \sum_{ I= (i_1 \cdots i_n) } a_I \cdot X_{i_1} \cdots X_{i_n} \ \Bigl| \ \mathrm{For \ any \ indexes} \ J \mathrm{\ and \ } K, \ \ \ a_J \cdot a_K =\sum_{L \in \mathrm{Sh}(J,K)} a_L \in \Z \ \Bigr\} .
\end{equation}

Hence, compared with Lemma \ref{lemapp1},
the correspondence $1+ X_i \mapsto \Upsilon_m( x_i) $ yields the isomorphism
$\mathrm{Im}\mathcal{M} \ra \mathrm{Im}(\Upsilon_m)$.
In conclusion, from \eqref{image}, we can characterize this $ \mathrm{Im}(\Upsilon_m) $ as a subgroup of $ GL_m(\Omega_m) $.

\section{Relation to quandle cocycle invariant.}\label{App2}
This section gives another diagrammatic computation of the Milnor invariants, in the sense of a quandle cocycle invariant \cite[\S 5]{CEGS}.

For this, we set up the map below \eqref{psidef}.
Using the explicit formula of $\mathrm{Im}(\Upsilon_m)$ in Lemma \ref{lemapp1} and the shuffle relation in \eqref{image},
we can concretely describe the set-theoretical section $\mathfrak{s}: \mathrm{Im}(\Upsilon_m)\ra \mathrm{Im}(\Upsilon_{m+1})$ (Here, the choice is a problem in the $(1,m+1)$th entry).
Then, according to $\epsilon \in \{ \pm 1\}$, let us define a map
\begin{equation}\label{psidef}
\phi^{\epsilon}_m : \mathrm{Im}(\Upsilon_m) \times \mathrm{Im}(\Upsilon_m) \lra GL_{m+1} (\Omega_{m+1} )
\end{equation}
by the gap between the section $\mathfrak{s}$ and conjugacy.
To be precise, we have
$$ \phi^{\epsilon}_m(A,B) = \left\{
\begin{array}{ll}
\mathfrak{s}(B)^{-1} \mathfrak{s}(A) \mathfrak{s}(B) \mathfrak{s}(B^{-1} AB)^{-1}, &\quad \mathrm{if \ } \epsilon =1, \\
\mathfrak{s}(B) \mathfrak{s}(A) \mathfrak{s}(B)^{-1} \mathfrak{s}(B AB^{-1})^{-1}, &\quad\mathrm{if \ } \epsilon =-1.
\end{array}
\right.$$
Since $ p_{m} \circ \phi^{\epsilon}_m $ is trivial, the image of $\phi^{\epsilon}_m$ is contained in the center $ \Ker (p_{m}) = \Upsilon_{m+1}( Q_{m+1}) . $

Next, we come to Proposition \ref{ppp2}.
Choose a link diagram $D$ of $L$, and suppose Assumption $\mathcal{A}_m$.
As in \S \ref{Sthm}, let us recall again the arcs $\alpha_k $ and $ \beta_k $, as well as the sign $\epsilon_k $ from Figure \ref{logifig},
Then, from the Wirtinger presentation, we can regard $f_m$ as a map $ \{ \ \mathrm{arcs \ of \ } D \ \} \ra \mathrm{Im}(\Upsilon_m) $.
We use \eqref{psidef} to define the product
\begin{equation}\label{fm+1}
\Phi_{m,j}(k) := \prod_{t : \ 1 \leq t \leq k } \phi^{\epsilon_t}_m \bigl( f_m(\alpha_t) , \ f_m (\beta_t)\bigr) \in \Upsilon_{m+1}( Q_{m+1}).
\end{equation}
According to \cite{CEGS}, this $\Phi_{m,j}(k) $ is called {\it the quandle cocycle invariant} (obtained from $p_m$).
\begin{prop}[{As a result in \cite[\S 5]{CEGS}}]\label{ppp2}
This $\Phi_{m,j}(k) $ is equal to $\Psi_{m} (j)$ in \eqref{fm+11}.
\end{prop}
We emphasize from the centrality of $Q_m$ that the sum formula is independent of the order of the crossings,
while the longitudes are seemingly non commutative
and that this $\Phi_{m,j}(k) $ is a diagrammatic computation of Milnor invariants.
However, the computation of $\Psi_{m,j}(k) $ is much faster than that of $\Phi_{m,j}(k) $ by definition.
In addition, we can similarly find a reduction of the higher $\bar{\mu}$-invariant in terms of quandle cocycle invariants,
although the formula is somewhat complicated.

\section{The first non-vanishing Milnor invariant of Milnor link.}\label{App3}
In this paper as in \S \ref{Sexa}, we emphasized that our computation of Milnor invariants does not need any description of longitudes $\mathfrak{l}_i$. However, in some cases, 
we can easily compute the invariants from describing $\mathfrak{l}_i$. 
As an example, we now compute the first non-vanishing Milnor invariant of the Milnor link.

Let $ E_m$ be the complement of Milnor link of $m$ components with $m>2$; see Figure \ref{bb}.
Choose the $k$-th longitude $\mathfrak{l}_k \in \pi_1(E_m )$ with $\mathrm{Ab}(\mathfrak{l}_{k})=0$.
We will determine the first non-vanishing Milnor invariant of $E_m$: the previous results are only for $ f_m(\mathfrak{l}_k )$ with $k=1$; see \cite{Mil2,HM}.

\begin{thm}\label{thm7} 
The invariant $ \Upsilon_m\circ f_m(\mathfrak{l}_k )_{( 1,m) }$ has the following form:
\begin{equation}\label{coco} (-1)^{m-k+1} \bigl[\bigl( [[\cdots[[\lambda^{(1)}_{1}, \lambda^{(2)}_{2}],\lambda^{(3)}_{3}]\cdots], \lambda_{k-1}^{(k-1)}]] \bigr),
\bigl( [[\cdots[[\lambda^{(m)}_k, \lambda^{(m-1)}_{k+1}], \lambda^{(m-2)}_{k+2}]\cdots],\lambda^{(k+1)}_{m}]\bigr) \bigr] .
\end{equation}
\end{thm}
Milnor invariants $f_m(\mathfrak{l}_k)$ with $k \neq 1$ of some links can sometimes be computed only from  
$f_m(\mathfrak{l}_1 )$ by using ``symmetry cyclic relation" in \cite{Mil1,Mil2}.
However, we can verify 
from the basis of $F_{m-1}/F_{m}$ (see Appendix \ref{App})
that if $ m \geq 4$, 
the above $ \Upsilon_m\circ f_m(\mathfrak{l}_k )_{( 1,m) }$ with $k \neq 1$
cannot be computed only from $ \Upsilon_m\circ f_m(\mathfrak{l}_1 )_{( 1,m) }$ by the linear independence.

The theorem directly follows from Lemmas \ref{lem8}--\ref{lem7} below. That is, it suffices to prove the lemmas.

\begin{lem}\label{lem8} 
Let $x_1, \dots, x_m$ be the arcs in Figure \ref{bb}. 
Then, the $(1,m)$-entry of
\begin{equation}\label{co2c} (-1)^{m-k} \Upsilon_m\bigl( \bigl[ \bigl([[[\cdots[[x_1,x_2],x_3]\cdots],x_{k-1}]\bigr) , \ \bigl([[[\cdots[[x_{m},x_{m-1}^{-1}],x_{m-2}^{-1}]\cdots],x_{k+1}^{-1}]\bigr)^{-1} \bigr] \bigr) \end{equation}
is equal to  \eqref{coco}.
\end{lem}
\begin{proof} We can verify the equality by induction on $k$ and $m$.
\end{proof}
\begin{lem}\label{lem7}
Let $m\geq 3$, and $k \in \mathbb{N}$ be $ k < m-1$.
Fix the meridians $x_1, \dots, x_m \in \pi_1(E_m )$ as in Figure \ref{bb}.
Recall the abbreviation $[g,h]= ghg^{-1}h^{-1}$.
Then, in the $(m+1)$-th quotient $ \pi_1(E_m )/ \Gamma_{m+1}\pi_1(E_m ) $ the longitudes are presented by
\begin{eqnarray*}
&\mathfrak{l}_{k} &=\bigl[ \bigl([[[\cdots[[x_1,x_2],x_3]\cdots],x_{k-1}]\bigr) , \ \bigl([[[\cdots[[x_{m},x_{m-1}^{-1}],x_{m-2}^{-1}]\cdots],x_{k+1}^{-1}]\bigr)^{-1} \bigr]^{-1}, \\
&\mathfrak{l}_{m} &=[[[\cdots [[x_{1},x_{2}],x_{3}]\cdots],x_{m-2}],x_{m-1}],\\
&\mathfrak{l}_{m-1} &=[[[\cdots[[x_{1},x_{2}],x_{3}]\cdots],x_{m-2}],x_{m}].
\end{eqnarray*}
\end{lem}
Here, it is worth noticing that the Milnor
link with 3 components is merely the Borromean rings. In addition, Milnor links with more components are inductively built from the Borromean rings by the Bing double operation \cite{B}. The following proof is analogous to the building.
\begin{proof}
Consider the right hand sides as the elements in $\pi_1(E_m)$, and denote them by $\mathfrak{r}_k$, $\mathfrak{r}_m$ and $\mathfrak{r}_{m-1}$,respectively.
Let $g_k \in \pi_1(E_m )/ \Gamma_{m+2}\pi_1(E_m ) $ be $\mathfrak{r}_k^{-1 }\mathfrak{l}_k$.
Here, it is enough to show that $g_k$ lies in $\Gamma_{m}\pi_1(E_m)$ and
is contained in the normal closure, $\langle x_k \rangle $, of $x_k$.
The proof is by induction on $m$. 
Since the proof for $m=3$ can be directly obtained from the Wirtinger presentation, we may assume $m>3$.

First, we focus on $\mathfrak{l}_k $ with $k <m$.
Consider a canonical solid torus $V \subset S^3$, which contains the $m$-th and $(m+1)$-th components.
Since $ E_{m+1} \setminus V $ is isotopic to $ E_m$, the inclusion $ E_{m+1} \setminus V \hookrightarrow E_{m+1} $ induces $\iota: \pi_1(E_m ) \ra \pi_1(E_{m+1} ) $.
If we replace the meridians $ x_k$ in $ \pi_1(E_{m+1})$ by $x_k'$, 
we have $ \iota( x_m)= [x_{m+1}', x_m'^{-1}]$ from the Wirtinger presentation.
Further, notice that $\iota( \mathfrak{l}_k)= \mathfrak{l}_k'$.
Since $g_k \in \Gamma_{m+1}(\pi_1 (E_m))\cap \langle x_k \rangle $ by assumption, $g_k'=\iota (g_k)$ is contained in $ \Gamma_{m+2}(\pi_1 (E_{m+1}))\cap \langle x_k' \rangle$.
Hence, the presentation of $ \mathfrak{l}_{k}'$ is equal to $ \mathfrak{l}_{k}$ by
replacing $ x_m$ by $[x_{m+1}',x_m'^{-1}]$, which is exactly the desired one on $m+1$.

Finally, we examine $\mathfrak{l}_{m-1}$ and $\mathfrak{l}_m$.
Similarly, we can find a solid torus $V' \subset S^3$, which contains the first and second components of $E_{m+1}$, such that
the inclusion $ E_{m+1} \setminus V' \hookrightarrow E_{m+1} $ yields 
a homomorphism $\kappa :\pi_1(E_m) \rightarrow \pi_1(E_{m+1})$
such that $\kappa ( x_1) =[x_1',x_2']$ and $\kappa ( x_t) =x_{t+1}$ for $t>1$. The remaining part of the proof goes as before. 
\end{proof}

\vskip 1pc

\normalsize

Max Planck Institute for Mathematics,
Vivatsgasse 7, 53111, Bonn, Germany

\

\normalsize
DEPARTMENT OF MATHEMATICS TOKYO INSTITUTE OF TECHNOLOGY
2-12-1
OOKAYAMA
, MEGURO-KU TOKYO
152-8551 JAPAN

\end{document}